\documentclass[11pt]{amsart}

\pdfoutput=1

\calclayout

\usepackage{amsfonts,amsmath}
\usepackage{graphicx}
\usepackage{bm}
\usepackage{comment}

\usepackage{verbatim}

\newtheorem{theorem}{Theorem}[section]
\newtheorem{proposition}[theorem]{Proposition}
\newtheorem{lemma}[theorem]{Lemma}
\newtheorem{corollary}[theorem]{Corollary}

\newtheorem{example}[theorem]{Example}
\newtheorem{definition}[theorem]{Definition}

\begin{document}

\title{Algebraic Model Selection and Experimental Design in Biological Data Science 
}

\author{Anyu Zhang}
\address{Amazon.com, Inc, 500 Boren Ave N, Seattle, WA 98109}
\email{anyuzhan@amazon.com}
\author{Jingzhen Hu}
\author{Qingzhong Liang}
\address{Department of Mathematics, Duke University, Durham, NC 27708}
\email{jingzhen.hu@duke.edu, qingzhong.liang@duke.edu}
\author{Elena S. Dimitrova}
\address{Department of Mathematics, California Polytechnic State University, San Luis Obispo, CA 93407}
\email{edimitro@calpoly.edu}
\author{Brandilyn Stigler}
\address{Department of Mathematics, Southern Methodist University, Dallas, TX 75275} \email{bstigler@smu.edu}

\thanks{All authors, except QL, were supported by National Science Foundation Awards 1419023, 1419038, 1720335, 1720341.}

\begin{abstract}
Design of experiments and model selection, though essential steps in data science, are usually viewed as unrelated processes in the study and analysis of biological networks. Not accounting for their inter-relatedness has the potential to introduce bias and increase the risk of missing salient features in the modeling process.
We propose a data-driven computational framework to unify experimental design and model selection for discrete data sets and minimal polynomial models. We use a special affine transformation, called a linear shift, to provide both the data sets and the polynomial terms that form a basis for a model. This framework enables us to address two important questions that arise in biological data science research: finding the data which identify a set of known interactions and finding identifiable interactions given a set of data. We present the theoretical foundation for a web-accessible database. As an example, we apply this methodology to a previously constructed pharmacodynamic model of epidermal derived growth factor receptor (EGFR) signaling.



\end{abstract}

\maketitle

\section{Introduction}

As technologies to collect, store, compute, and analyze large data sets from laboratory experiments have become routine in biomedical research, biological data science \cite{schatz} has emerged as a new member of an array of disciplines, where the emphasis is on data-driven methods.
Examples include topological data analysis and persistent homology, machine learning, as well as a variety of algebraic approaches~\cite{carlsson20, chazal16, Chazal2017AnIT,deisenroth2020mathematics, Ghrist2018HomologicalAA, wasserman18}.
What all of these methods have in common is being a component of the data-science pipeline: data collection and processing which includes experimental design; modeling, analysis, and visualization which include model selection and reduction; and finally decision-making.

While these components are traditionally viewed as separate processes, there have been efforts to couple them.
In~\cite{jeong} the authors describe a framework for connecting design of experiments and model reduction to address the problem of ``bridging the gap between high model complexity and
limited available experimental data.''  However, limited data lead to a very large pool of candidate models which impedes the model selection process.
In~\cite{robbiano} the author connects experimental design and model selection for discrete data and models given by polynomials.
It is stated that a polynomial $f$ is \emph{identifiable} by a data set $D$ if and only if the matrix formed by evaluating the monomials on the data is full rank.
While the condition for identifiability was given, the strategy for choosing data points was not provided.  Furthermore, no procedure to construct a polynomial that is identifiable by a given set of points was given.  In short, both the data points and the monomials comprising the polynomial must be known in order to decide on identifiability.  As a third example, in~\cite{dimitrova2014data} the authors specialized the results in~\cite{robbiano} to so-called \emph{minimal} polynomial models which allowed them to determine identifiability without knowledge of the monomials.

In this work we propose a framework that unifies experimental design and model selection for discrete data sets and minimal polynomial models. The framework is built on affine transformations of input data sets, called \emph{linear shifts}.
Input data may be thought of as initializations or stimuli of nodes in a network, while the corresponding output data may be interpreted as the trajectory of the initialization or the response of the network to the stimuli.
%
\emph{Minimal} here refers to no subpolynomials which correspond to unobserved interactions.
We provide both the data sets and the monomials that form a basis for a model: input data are paired with model bases that are identifiable by the input data so that model bases and data points are chosen simultaneously.

Such a framework can address the following types of questions:
\begin{itemize}
    \item Given a known interaction, which experiments/data identify the interaction?
    \item Given a set of data, corresponding to experimental conditions, what interactions are identifiable?
\end{itemize}
An outcome of this framework  is efficiency in moving through a data-science pipeline: iterations between experimental design and model selection with incremental progress is replaced with simultaneous execution with an exhaustive view of all possible predictive outcomes.

In Section \ref{sec:background} we provide relevant background including prior results on linear shifts. What follows is a theoretical development of equivalence classes defined by linear shifts. In particular in Section \ref{sec:results}, we define the representative of an equivalence class, which may be used to display data in a ``standard position'' and establish characteristics of the representatives as well as interesting divisibility properties of the equivalence classes.
The theoretical results are accompanied by computational tools in the form of a database and a website, described in Section \ref{sec:web-introduction}.  We demonstrate how to address the guiding questions on a signalling network using the proposed framework in Section \ref{sec:applications} and close with a discussion of implications of this work.

\section{Background}
\label{sec:background}

Let $n\in\mathbb N$ be the number of coordinates representing nodes in a network.  We denote by $p\in\mathbb N$ the number of states each node can take and view the states as elements of $\mathbb Z_p=\{0,\ldots,p-1\}$, that is the set of integer remainders upon division by $p$. In order for~$\mathbb Z_p$ to assume the structure of a finite field, we require that $p$ be prime\footnote{While field theory permits  prime-powered cardinality, we restrict to the prime case to facilitate interpretation of the values in the field.}. Finally, let $m\in\mathbb N$ denote the  number of input data points, each of which corresponds to an initialization  or stimulus of the network.

Let $\left\{\mathbb Z_p^n\right\}_m$  denote the collection of all possible subsets of $\mathbb Z_p^n$ of $m$ points.
Note that $\big|\left\{\mathbb Z_p^n\right\}_m\big| = \binom{p^n}{m}.$ The input sets we consider  are subsets of $\left\{\mathbb Z_p^n\right\}_m$ for fixed $p$, $n$, and $m$.

\begin{example}
	Let $p=3$, $n=2$, and $m=2$. Then $\left\{\mathbb Z_3^2\right\}_2 = \{\{(0,0),(0,1)\}$, $\{(0,0),(0,2)\}$, $\{(0,0),(1,0)\}$, $\{(0,0),(1,1)\}$, $\{(0,0),(1,2)\},\ldots,\{(2,1),(2,2)\}\}$. The  values 0, 1, and 2 may be interpreted as low, medium, and high, respectively.  The set $\{(0,0),(2,2)\}$ may represent two initializations of a network: in one experiment both nodes are low and in the other  both are high.
\end{example}

A polynomial $f$ is \emph{identifiable} by a set $S$ of data if and only if the evaluation matrix $\mathbb X(f,S)$ is full rank, where the rows are the points in $S$, the columns are the monomials that appear in $f$, and the entries are evaluations of the points on the monomials~\cite{robbiano}.

\subsection{Algebraic Geometry Preliminaries}

Most of the terms, notation, and language in this section are based on \cite{cox}.

Given a set $S\subseteq \mathbb{Z}_p^n$ of input data, it is of interest to consider all polynomials in $n$ variables with coefficients in $\mathbb Z_p$ which vanish on $S$. We call the set of these polynomials an \emph{ideal of points}, denoted $I(S)$. Since for multivariate polynomials there is no unique way of ordering the polynomials terms (\emph{e.g.} consider $x^2$ and $xy$), we need to make a choice for the ordering of the monomials in a given polynomial. We thus introduce a \emph{monomial order} $\prec$ (sometimes called a term order or an admissible order) to be a total order on the set of all (monic) monomials.

The choice of a monomial order $\prec$ allows for sorting the terms of a polynomial. The leading term of a polynomial $f$ is thus the term of the largest monomial for the chosen monomial order, denoted as $LT_\prec(f)$. Also we call $LT_\prec(I(S))=\langle LT_\prec(f):f\in I(S)\rangle$ the \emph{leading term ideal} for an ideal $I(S)$.

Now let $\prec$ be a monomial order and let $I(S)$ be an ideal of points. Then $G\subseteq I(S)$ is a \emph{Gr\"obner basis (GB)} for $I(S)$ with respect to $\prec$ if for all $f\in I(S)$ there exists $g\in G$ such that the leading term $LT_\prec(g)$ divides $LT_\prec(f)$. Also $G$ is \emph{reduced} if every $g\in G$ is monic and no monomials in $g-LT_\prec(g)$ are leading terms.

\begin{definition}
The monomials which do not lie in $LT_\prec(I(S))$ are called the \emph{standard monomials} (SMs) of $I(S)$ with respect to $\prec$, denoted $SM_\prec(I(S))$.
\end{definition}

Gr\"obner bases exist for every nonzero ideal of points and every $\prec$ and make multivariate polynomial division well defined in that remainders are unique; furthermore for any ideal the set of reduced GBs is finite. Let~$R$ be the set of all polynomials in $n$ variables over $\mathbb Z_p$. Then $R/I(S)=\{f+I(S):f\in R\}$ is the set of polynomial functions that are ``minimal'' with respect to $I(S)$ in the sense that there is no nonzero polynomial $g\in R$ and $h\in R$ such that $f = h + g$ and $g$ is identically zero on the points in $S$. Any set of standard monomials $SM_\prec(I(S))$ for a given monomial order forms a basis for $R/I(S)$ as a vector space with dimension $|S|$. Furthermore, sets of standard monomials are in one-to-one correspondence with leading term ideals for $I(S)$.

Sets of standard monomials are said to be \emph{closed} as they satisfy the following: if $x^\alpha \in SM_\prec (I(S))$ and $x^\beta$ divides $x^\alpha$, then $x^\beta \in SM_\prec (I(S))$.  This divisibility property on monomials is equivalent to the following geometric relation on data points plotted on a lattice.

\begin{definition}
\label{def:staircase}
A set $\lambda\subset \mathbb N^n$ is a \emph{staircase} if  for all $u\in \lambda$, $v\in \mathbb N^n$ and $v\leq u$ imply $v\in \lambda$.
\end{definition}
The relation $v\leq u$ is defined coordinatewise, in that $v_i\leq u_i$ for $1\leq i\leq n$.

Once experiments are performed using the points in a given input data set $S$ as initializations, for each node $x_i$, $i=1,\ldots n$, we have an input-output data set $D_i=\{(\mathbf{d_1},t_{i1}),\ldots,\allowbreak(\mathbf{d_m},t_{im})\}$,
where $\mathbf{d_j}\in S$  are network inputs and $t_{ij}\in \mathbb{Z}_p$ are the experimental outputs.
Consider the $n$-tuple $F=(f_1,\ldots ,f_n)$, where each coordinate is a polynomial function $f_i:\mathbb{Z}_p^n\rightarrow \mathbb{Z}_p$ which \emph{fits} $D_i$ in the sense that $f_i(\mathbf{d_j})=t_{ij}$ for each $j=1,\ldots, m$.
The \emph{model space} for the input-output data sets $D_1,\ldots, D_n$ is the set $F+I(S):=\{(f_1+h_1,\ldots,f_n + h_n):h_i\in I(S)\}$ of all systems of polynomials which fit the data in $D_1,\ldots, D_n$. A particular model from the model space can be selected by choosing a monomial order $\prec$ and then computing the remainder of each $f_i$ when divided by the polynomials in $I(S)$ written in terms of the GB for $I(S)$ with respect to $\prec$.  We call
$$(f_1\mathrm{~mod~} G,\ldots,f_n\mathrm{~mod~} G)$$
the \emph{minimal} model with respect to $\prec$, where $G$ is a GB for $I(S)$ with respect to $\prec$, since no sum of nonzero terms of $f_i\mathrm{~mod~} G$ vanishes on $S$. An algorithm for computing all minimal models for a given input-output data set was first given in~\cite{laubenbacher}.

The following example illustrates that each choice of GB results in a different minimal model.

\begin{example}
\label{ex:example1}
Consider $S_1=\{(0,0),(1,1)\}\subseteq \mathbb Z_2^2$, where 0 and 1 may be interpreted as off and on, respectively. The  ideal $I(S_1)$ of the points in $S_1$ has two reduced GBs, namely 
\begin{center}
$G_1=\{\underline{x_1}-x_2,\underline{x_2^2}-x_2\}\mathrm{~and~} G_2=\{\underline{x_2}-x_1,\underline{x_1^2}-x_1\}$,
\end{center}
where the leading terms are underlined. The corresponding standard monomial sets are $\{1, x_2\}$ and $\{1, x_1\}$, respectively. As such there are two resulting minimal models:
any minimal model with respect to $G_1$ will be a linear combination of 1 and $x_2$ as all $x_1$'s cancel out. Similarly, any minimal model with respect to $G_2$ will be in terms of 1 and $x_1$ only.

If $S_2=\{(0,0),(0,1)\}$, then $I(S_2)$ has a unique GB
$\{\underline{x^2_2}-x_2,\underline{x_1}\}$, corresponding to a unique standard monomial set $\{1,x_2\}$, and resulting in a unique minimal model. Also $S_3=\{(1,0),(1,1)\}$  has a unique associated GB
$\{\underline{x^2_2}-x_2,\underline{x_1}+1\}$, corresponding to the same standard monomial set $\{1,x_2\}$ as $S_2$.
\end{example}

Notice that the last two input data sets in the above example return the same standard monomial basis. Is there any relationship between the two sets? We answer this next.

\subsection{Linear Shifts of Input Data Sets}
\label{sec:LS}

Notice that we can construct the following affine transformation:

\[
\begin{bmatrix}
a_1  \hspace{0.1cm}    0 \hspace{0.25cm}b_1\\

\hspace{0.1cm} 0 \hspace{0.1cm}       \hspace{0.1cm} a_2 \hspace{0.1cm} b_2 \\
\end{bmatrix}
\begin{bmatrix}
    0  \hspace{0.1cm}    \hspace{0.1cm} 0\\
    0  \hspace{0.1cm}    \hspace{0.1cm} 1\\
    1  \hspace{0.1cm}    \hspace{0.1cm}  1         \\
\end{bmatrix}
=
\begin{bmatrix}
1  \hspace{0.1cm}    0 \hspace{0.25cm}1\\

0 \hspace{0.1cm}     1 \hspace{0.25cm} 0 \\
\end{bmatrix}
\begin{bmatrix}
    0  \hspace{0.1cm}    \hspace{0.1cm} 0\\
    0  \hspace{0.1cm}    \hspace{0.1cm} 1\\
    1  \hspace{0.1cm}    \hspace{0.1cm}  1         \\
\end{bmatrix}
=
\begin{bmatrix}
    1  \hspace{0.1cm}   \hspace{0.1cm} 1\\
    0 \hspace{0.1cm}    \hspace{0.1cm} 1\\
\end{bmatrix}
.\]

By applying the function $f_1$ to the first coordinate and $f_2$ to the second coordinate of the points in the input data set $S_2=\{(0,0), (0,1)\}$, they will ``shift'' the set to $S_3=\{(1,0), (1,1)\}$.  This affine transformation on input data sets was termed a \emph{linear shift} in~\cite{he2016}.

As the model space can be large, this creates ambiguity in predictions. We can eliminate this ambiguity by reducing the number of GBs.  While our ultimate goal is to identify input data sets with a unique GB, we propose a method for deleting redundant information which is motivated by the linear shift transformation. By checking if an unknown data set is the linear shift of a known data set, then the associated standard monomial bases will be found without additional computational cost.

\begin{definition}[\cite{he2016}]
\label{def:LS}
Let $S_i=\{\mathbf{x}^i_1,\ldots,\mathbf{x}^i_m\},S_j=\{\mathbf{x}^j_1,\ldots,\mathbf{x}^j_m\}\in \mathbb Z_p^n$ be input data sets, i.e.
$$S_i=\{(x_{11}^i,x_{12}^i,\ldots,x_{1n}^i),\ldots,(x_{m1}^i,x_{m2}^i,\ldots,x_{mn}^i)\}$$ and $$S_j=\{(x_{11}^j,x_{21}^j,\ldots,x_{1nn}^j),\ldots,(x_{m1}^j,x_{m2}^j,\ldots,x_{mn}^j)\}.$$

We say that $S_i$ is a \emph{linear shift} of $S_j$, denoted $S_i\sim S_j$, if there exists $\phi=(\phi_1,\dots,\phi_n):\mathbb Z_p^n\rightarrow \mathbb Z_p^n$ such that
 $\phi_k(\mathbf{x}^i_{uk})=a_k{\mathbf{x}^i_{uk}}+b_k=\mathbf{x}^j_{uk}$, $a_k\in \mathbb Z^\times_p, b_k\in \mathbb Z_p, $  $k=1,\ldots,n$, and $u = 1,\ldots, m$, i.e. $S_i=\phi(S_j)$.
\end{definition}

Throughout the rest of the manuscript, we will denote input data sets as in Definition~\ref{def:LS}.

\begin{example}
    Consider the linear shift illustrated in Figure \ref{fig:lin-shift} for $p=3, n=2, m=3$. There are data sets for which there is no linear shift between them;  see Figure \ref{fig:no-lin-shift}.
\end{example}

It is straightforward to see that the linear shift is a bijection and thus induces an equivalence relation on $\left\{\mathbb Z_p^n\right\}_m$ \cite{he2016}. The following matrix encodes the equivalence class associated to an input data set.

\begin{figure}
\label{fig:lin-shift}
\begin{center}
\includegraphics[width=9cm]{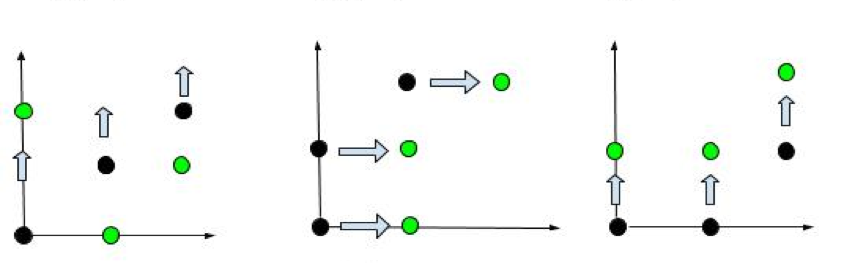}
\caption{Linear shifts of points in $\mathbb Z_3^2$. Left: The data set of black points  $\{(0,0),(1,1),(2,2)\}$  is the linear shift of the data set of green points $\{(0,2),(1,0),(2,1)\}$ by  $\phi=(\phi_1,\phi_2)=(x,y+2)$.
Center: The  set of black points  $\{(0,0),(0,1),(1,2)\}$ is the linear shift of the green points $\{(1,0),(1,1),(2,2)\}$ by  $\phi=(\phi_1,\phi_2)=(x+1,y)$.
Right: The  set of black points  $\{(0,0),(1,0),(2,1)\}$ is the linear shift of  the green points $\{(0,1),(1,1),(2,2)\}$ by  $\phi=(\phi_1,\phi_2)=(x,y+1)$.}
\end{center}
\end{figure}

\begin{figure}
\label{fig:no-lin-shift}
\begin{center}
\includegraphics[width=3cm]{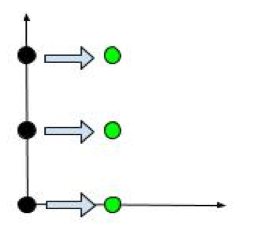}
\includegraphics[width=3cm]{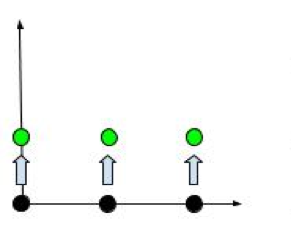}
\caption{The data set of black points  $\{(0,0),(0,1),(0,2)\}$ in the left plot is not the linear shift of the set of black points $\{(0,0),(1,0),(2,0)\}$ in the right plot: there is no map for the first coordinate from a single value~0 in the data set on the left to three  values in the data set on the right.}
\end{center}
\end{figure}

\begin{definition}
\label{def:LSM}
	Let $S$ be an input set with fixed $p$, $n$, and $m$. We define the \emph{linear shift matrix} as
	\begin{displaymath}
\mathcal{M}(S) =  \left(\begin{array}{cccc}
a_1x_{11}+b_1 & a_2x_{21}+b_2 & \cdots &  a_nx_{n1}+b_n\\
\vdots &\vdots &\ddots &\\
a_1x_{1m}+b_1 & a_2x_{2m}+b_2 & \cdots & a_nx_{nm}+b_n
\end{array}\right)\in \left(\left\{\mathbb Z_p^n\right\}_m\right)^{p^n\times(p-1)^n} \end{displaymath}
in which each column corresponds to a choice of $\bm{a} = (a_1,\ldots,a_n)$, and each row corresponds to a choice of $\bm{b}=(b_1,\ldots,b_n)$, where $a_i \in \mathbb{Z}_p \setminus \{0\}$ and $b_j \in \mathbb{Z}_p$ for $1\le i,j  \le n$.
We use $$\mathcal{M}(S,\bm{a},\bm{b}) = \{(a_1x_{11}+b_1,\ldots,a_nx_{1n}+b_n) , \ldots, (a_1x_{m1}+b_1,\ldots,a_nx_{mn}+b_n)\}$$ to denote  an entry corresponding to column $\bm{a}$ and row $\bm{b}$ of the linear shift matrix.
\end{definition}

Note that $\mathcal{M}(S,\bm{1},\bm{0}) =S$.

\begin{example}
	Let $S = \{(0,1),(1,0)\}\in \left\{\mathbb Z_3^2\right\}_2$. The linear shift matrix $\mathcal{M}(S)$ is represented  in Example \ref{example matrix 1}, where the entries comprise the equivalence class $[S]$, namely

\noindent\begin{tabular}{rl}
 & $\{\{(0, 0), (1, 1)\},
\{(0, 0), (1, 2)\},
\{(0, 0), (2, 1)\},
\{(0, 0), (2, 2)\},
\{(0, 1), (1, 0)\},
\{(0, 1), (1, 2)\}$,\\
 &  $\{(0, 1), (2, 0)\},
\{(0, 1), (2, 2)\},
\{(0, 2), (1, 0)\},
\{(0, 2), (1, 1)\},
\{(0, 2), (2, 0)\},
\{(0, 2), (2, 1)\}$,\\
&  $\{(1, 0), (2, 1)\},
\{(1, 0), (2, 2)\},
\{(1, 1), (2, 0)\},
\{(1, 1), (2, 2)\},
\{(1, 2), (2, 0)\},
\{(1, 2), (2, 1)\}\}.$
\end{tabular}

    \begin{table}[ht]
	\centering
	\begin{tabular}{ c | c c c c}
		& $a_1=1,a_2=1$ & $a_1=1,a_2=2$ & $a_1=2,a_2=1$ & $a_1=2,a_2=2$ \\ \hline
		$b_1=0,b_2=0$ & $\{(0,1),(1,0)\}$ & $\{(0,2),(1,0)\}$ & $\{(0,1),(2,0)\}$& $\{(0,2),(2,0)\}$ \\
		$b_1=0,b_2=1$ & $\{(0,2),(1,1)\}$ & $\{(0,0),(1,1)\}$ & $\{(0,2),(2,1)\}$& $\{(0,0),(2,1)\}$   \\
		$b_1=0,b_2=2$ & $\{(0,0),(1,2)\}$ & $\{(0,1),(1,2)\}$ &$\{(0,0),(2,2)\}$& $\{(0,1),(2,2)\}$   \\
		$b_1=1,b_2=0$ & $\{(1,1),(2,0)\}$ & $\{(1,2),(2,0)\}$ & $\{(1,1),(0,0)\}$& $\{(1,2),(0,0)\}$   \\
		$b_1=1,b_2=1$ & $\{(1,2),(2,1)\}$ & $\{(1,0),(2,1)\}$ & $\{(1,2),(0,1)\}$ & $\{(1,0),(0,1)\}$ \\
		$b_1=1,b_2=2$ &$\{(1,0),(2,2)\}$ & $\{(1,1),(2,2)\}$ & $\{(1,0),(0,2)\}$& $\{(1,1),(0,2)\}$ \\
		$b_1=2,b_2=0$ &$\{(2,1),(0,0)\}$ & $\{(2,2),(0,0)\}$ &$\{(2,1),(1,0)\}$ & $\{(2,2),(1,0)\}$ \\
		$b_1=2,b_2=1$ &$\{(2,2),(0,1)\}$ & $\{(2,0),(0,1)\}$ & $\{(2,2),(1,1)\}$& $\{(2,0),(1,1)\}$  \\
		$b_1=2,b_2=2$ &$\{(2,0),(0,2)\}$ & $\{(2,1),(0,2)\}$ &$\{(2,0),(1,2)\}$& $\{(2,1),(1,2)\}$
	\end{tabular}
	\caption{The linear shift matrix $\mathcal{M}(S)$ for $p=3,n=2,m=2$ with input set $S=\{(0,1),(1,0)\}\in  \left\{\mathbb Z_3^2\right\}_2$. Each entry of the matrix is the linear shift $(a_1x+b_1,a_2y+b_2)$ of the set $S$. For example, in the third row and the second column, the linear shift  $(x,2y+2)$ of $S$ is $\{(0,2\cdot 1+2),(1,2\cdot 0+2)\}=\{(0,1),(1,2)\}$.
    \label{example matrix 1}}
    \end{table}
\end{example}

\section{Properties of the Equivalence Classes Defined by Linear Shifts}
\label{sec:results}

In previous work, linear shifts of input data sets of the same size were shown to partition these sets into equivalence classes~\cite{he19}.  Moreover, those equivalence classes are annotated with bases of monomials that are identifiable by any input data set in the equivalence class.  Since each choice of basis results in a unique model, we refer to the monomial bases as \emph{model bases}. As stated in the Introduction, we propose a framework for connecting experimental design and model selection. The foundation for this framework is the set of equivalence classes of input data sets annotated by model bases.  In this section, we present properties of these annotated equivalence classes for data sets in $\left\{\mathbb Z_p^n\right\}_m$, where $p$ is the number of states, $n$ is the number of coordinates, and~$m$ is the number of input data points.

\subsection{Representatives}
\label{sec:representatives}
We now define representatives for the equivalence classes of input data sets of same size.  We use a distance function to identify sets that are the most ``condensed,'' in that their elements are the closest to the origin and so maximize the number of zeros in the coordinates.

\begin{lemma}[Proposition 78 in~\cite{he2016}]
\label{thm:staircaseD}
If $S_1$ and  $S_2$  are distinct staircases for fixed $p, n$, and $m$, then for any monomial order $LT_\prec(I(S_1))\neq LT_\prec(I(S_2))$ with different corresponding standard monomial bases.
\end{lemma}

\begin{proof}
Note that linear shift is a bijection and it follows that it is an equivalence relation. Consider the equivalence class $[S]$ defined by a linear shift. Let
$[S]=\{S_{t_1},S_{t_2},...,S_{t_W}: S_{t_i}\sim S_{t_j},  \forall t_i,t_j \in \{t_1,t_2,...,t_W\} ,i\neq{j}\}$. Assume there are two distinct staircases, $S_i$ and $S_j$, with the same standard monomials and leading terms. Then there exists a point $u \in S_i$ such that $u \not\in S_j $. There are two cases to consider:

\begin{itemize}
   \item $u > \lambda$ for all $\lambda \in S_j$. Suppose the number of points in $ S_i$ is $|S_i|\geq |S_l|+1$, where $S_l=\{w \in S_i: w<u\}$. As $S_i$ is a staircase, $S_l$ will contain all the points below $u$. But for $u > \lambda$ for all $\lambda \in S_j$, we have that $|S_j|\leq |S_l|< |S_i|$.
So, $|S_i|\neq |S_j|$, contradicting the assumption that $S_i$ and $S_j$ have the same number of points.
   \item $u \leq \lambda$ for some $\lambda \in S_j$. However, by assumption $u \not \in S_j$ and so $S_j$ cannot be a staircase.
\end{itemize}

As both cases lead to contradictions, we conclude that distinct staircases have different leading terms and standard monomials.
\end{proof}



\begin{theorem}[Proposition 79 in~\cite{he2016}]
\label{thm:staircaseShift}
If  $S_1$ and $S_2$ are distinct staircases, then there is no linear shift between them.
\end{theorem}


\begin{proof}
Suppose $S_1$ and $S_2$ are distinct staircases.
By  Lemma \ref{thm:staircaseD}, $S_1$ and $S_2$ have different standard monomial bases and $LT_\prec(I(S_1))\neq LT_\prec (I(S_2))$.
But if there exists a linear shift between $S_1$ and $S_2$, then it must be that $LT_\prec(I(S_1))= LT_\prec(I(S_2)$ \cite{he2016}, hence a contradiction. Therefore, there is no linear shift between distinct staircases.
\end{proof}

Notice that Theorem \ref{thm:staircaseShift} implies that an equivalence class can contain 0 or 1 staircase.

\begin{corollary}
\label{thm:unique}
An equivalence class contains at most one staircase.
\end{corollary}


Since an equivalence class contains at most one staircase, staircases are a natural choice of representatives. The ``condensedness'' of the points in staircases motivates the following definition.

\begin{definition}
\label{def:distance}
Let $S\subseteq \mathbb{Z}_p^n$. The \emph{set distance} of $S$, denoted $D(S,0)$, is  the sum of the Euclidean distances of all points in $S$ to the origin.
\end{definition}

\begin{example}
\label{ex:unique-rep}
Consider the sets in Figure \ref{fig:no-lin-shift}. The set $S=\{(0,0),(0,1),(0,2)\}$ (left in black) is the representative for its equivalence class as it has the smallest set distance, namely $D(S,0)=0+1+2=3$. The set $T=\{(1,0),(1,1),(1,2)\}$ (left in green) has  set distance $D(T,0)=1+\sqrt 2+\sqrt 5>3$. Note that $S$ is a staircase.
\end{example}

We will see that staircases have minimum set distance; however, there are sets that are not staircases that have minimum set distance.

\begin{example}
Consider the sets $S=\{(0, 0, 0), (0, 0, 1), (0, 1, 0), (1, 1, 0)\}$ and $T=\{(0, 0, 0), \allowbreak (0, 1, 0), \allowbreak (0, 1, 1), \allowbreak (1, 0, 0)\}$ in $\mathbb Z_2^3$. The set distances of $S$ and $T$ are equal: $D(S,0)=D(T,0)=2+\sqrt 2$.  We will show that these two sets have minimum set distance.  In order for a set to have a smaller set distance than $S$ and $T$, its points must have fewer ones in the entries.  In fact the only set with smaller set distance is $U=\{(0, 0, 0), (0, 0, 1), (0, 1, 0), (1, 0, 0)\}$ with $D(U,0)=3$. However there is no linear shift between $S$ and $U$. While the identity map works for the first and third coordinates, no function exists for the second coordinate: $\{0,0,1,1\} \not\mapsto
\{0,0,1,0\}$. So there is no set with a smaller set distance that is a linear shift of either of them.  Hence $S$ and $T$ have minimum set distance.
\end{example}

In order to define a representative of an equivalence class, we must order sets which share the same set distance.  Let  $S$ and $T$ be set with the same set distance.  Order the points in each set using lexicographic order (with increasing index order).  We say that $S<T$ if $p_1=q_1,\ldots,p_{k-1}=q_{k-1}$ and $p_k\prec q_k$ in lexicographic order (with increasing index order), where $p_1,\ldots,p_k\in S$ and $q_1,\ldots,q_k\in T$.

\begin{definition}
 The set $S$ is the \emph{representative} for its equivalence class $\mathcal E$ if $S$ is the unique set that has the minimum set distance in $\mathcal E$; when $S$ is not unique we have that $S<T$ for all sets $T\in \mathcal E$ with $T\neq S$ and $T$ also has minimum set distance.
\end{definition}

We note that the choice of the representative is unique, which the following example illustrates.

\begin{example}
Consider the sets $S=\{(0, 0, 0), (0, 0, 1), (0, 1, 0), (1, 1, 0)\}$ and $T=\{(0, 0, 0), \allowbreak (0, 1, 0), \allowbreak (0, 1, 1), \allowbreak (1, 0, 0)\}$ in $\mathbb Z_2^3$. We showed in the previous example that $S$ and~$T$ have minimum set distance.  As $S<T$, we have that $S$ is the representative of its equivalence class. Note that $S$ is not a staircase as $(1,0,0)<(1,1,0)$ but $(1,0,0)\not\in S$.
\end{example}

\begin{theorem}
If an equivalence class contains a staircase, then the staircase is the unique representative.
\end{theorem}

\begin{proof}
Suppose $S_1=\{(x_{11}^1,\ldots,x_{1n}^1),\ldots, (x_{m1}^1,\ldots,x_{mn}^1)\}$ is a staircase of $m$ points and $S_2=\{(x_{11}^2,\ldots,x_{1n}^2),\ldots,(x_{m1}^2,\ldots,x_{mn}^2)\}$ is in the same equivalence class. Then the set distance of $S_2$ to the origin is
\begin{align*}
D(S_2,0) &=
\sum_{t=1}^{m} \sqrt{(x_{t1}^2)^2+\cdots +(x_{tn}^2)^2} \geq\sum_{t=1}^{m}\sqrt{(x_{t1}^2-x_{t_11}^2)^2+\cdots +(x_{tn}^2-x_{t_1n}^2)^2}\\
&\geq \sum_{t=1}^{m}\sqrt{(a_1x_{t1}^1+b_1-b_1)^2+\cdots +(a_n x_{tn}^1+b_n-b_n)^2}
>\sum_{t=1}^{m}\sqrt{(x_{t1}^1)^2+\cdots +(x_{tn}^1)^2} \\
& > D(S_1,0).
\end{align*}

Here, $a_i$ and $b_i$ are the linear shift components from $S_1$ to $S_2$, and $(x_{t_11}^2,\ldots, x_{t_1n}^2)=(b_1,\ldots,b_n)$ is the shifted origin point in $S_2$. We know that the distance of $S_1$ to the origin will always be smaller than the distance of $S_2$ to the origin for $a_i\neq 0$ for $i=1,\ldots,m$. Therefore, $S_1$ is the representative in its equivalence class and we already know from Corollary \ref{thm:unique} that there is only one staircase in each equivalence class. It follows that $S_1$ as a staircase is the representative. Since equivalence classes contain at most one staircase and the set distance of a staircase is strictly less than the set distances for all other sets in the equivalence class, this means that the staircase is the unique representative.
\end{proof}

\begin{example}
The set $S$ in Example \ref{ex:unique-rep} is a staircase and is therefore the representative of its equivalence class.
\end{example}

The definition of representative offers two valuable perspectives.  First, the property of the points in a representative set being closest to the origin has an interpretation in the context of experimental design: such points can be viewed as experimental settings with the fewest active nodes.  In systems that permit such initializations, these experiments call for all nodes to be turned off except for a critical few.  However, in systems for which it is infeasible to turn many nodes off, it is advantageous to have available other initializations that are associated with the same bases, i.e. possible interactions. The hope is that at least one of the initializations in an equivalence class is feasible.

Second, there is also a geometric interpretation: the configuration of the points in a representative set can be viewed as the standard position for sets of points in the equivalence class. This is analogous to writing equations of geometric shapes, such as ellipses, in standard form.  Then any geometric manipulations of the data, such as stretching, which can be described via linear shifts, do not change the associated model bases.

\subsection{Divisibility Properties of Equivalence Classes}
\label{sec:div-eq}

In this section, we describe some divisibility properties of equivalence classes and conclude with an upper bound for the number of equivalence classes. The main results are proved here but some of these proofs rely on results proved in the Appendix (Section \ref{apdx:div}).

We begin by recalling that according to~\cite{he2016}, if $S_1\sim S_2$, then $LT_\prec(I(S_1))=LT_\prec(I(S_2))$ and $SM_\prec(I(S_1))=SM_\prec(I(S_2))$ for any term order~$\prec$.
However, notice that if $SM_\prec(S_1)=SM_\prec(S_2)$, it does not follow that $S_1\sim S_2$ as the next example illustrates.

\begin{example}
Over $\mathbb Z_7$ the sets $\{(0,0),(1,0),(2,0)\}$ and $\{(0,0),(1,0),(3,0)\}$ have the same set of standard monomials, namely $\{1,x,x^2\}$, but there is no linear shift correspondence between them.
\end{example}

The following few results describe the structure of the linear shift matrix which was introduced in Section \ref{sec:background}.

\begin{lemma}
\label{thm:uniqueness inside column}
Let $p\not|\; m$. In each column of a linear shift matrix the data sets are distinct.
\end{lemma}

\begin{proof}
Let $n \in \mathbb Z^+$ and $p$ be prime. Let $m \in \mathbb Z^+$, $m \leq p^n$,  and $p\not|\; t$. Suppose $S$ is a data set, $S =  \{(x_{11},x_{12},\ldots,x_{1n}) , \ldots, (x_{m1},x_{m2},\ldots,x_{mn})\} $.
Assume for the sake of contradiction that $\mathcal{M}(S,\bm{a},\bm{\hat{b}}) = \mathcal{M}(S,\bm{a},\bm{b})$, where $\bm{a}=(a_1,\cdots,a_n)$, $\bm{\hat{b}} =(\hat{b}_1,\cdots,
\hat{b}_n)$, $\bm{b} = ({b}_1,\cdots,{b}_n)$, $a_i \in \mathbb Z_p \backslash \{0\}$, $b_i, \hat{b_i} \in \mathbb Z_p$, $1\leq i  \leq n$, and $\bm{\hat{b}} \neq \bm{{b}}$. The explicit expression is
 \begin{equation*}
 \begin{array}{ccc}
 \{(a_1x_{11}+\hat{b}_1,a_2x_{12}+\hat{b}_2,\cdots,a_nx_{1n}+\hat{b}_n), & \cdots,& (a_1x_{t1}+\hat{b}_1,a_2x_{t2}+\hat{b}_2,\cdots,a_nx_{tn}+\hat{b}_n)\} \\
 & \parallel  &  \\
 \{(a_1x_{11}+{b}_1,a_2x_{12}+{b}_2,\cdots,a_nx_{1n}+{b}_n), & \cdots, & (a_1x_{t1}+{b}_1,a_2x_{t2}+\overline{b}_2,\cdots,a_nx_{tn}+{b}_n)\}.
 \end{array}
 \end{equation*}
Adding all elements inside the output, the summations are identical as every element in one output has a corresponding identical one in the other output. Then, we obtain
 \begin{equation*}
 \begin{array}{cccc}
 \left(a_1\sum_{i=1}^{m}x_{i1}+m\hat{b}_1, \right.& a_2\sum_{i=1}^{m}x_{i2}+m\hat{b}_2,&\cdots,&\left.a_n\sum_{i=1}^{t}x_{in}+m\hat{b}_n\right)\\
 & \parallel  &  \\
 \left(a_1\sum_{i=1}^{m}x_{i1}+m{b}_1,\right.&a_2\sum_{i=1}^{m}x_{i2}+m{b}_2, & \cdots, &\left. a_n\sum_{i=1}^{t}x_{in}+m{b}_n\right),
 \end{array}
 \end{equation*}
 which leads to $\hat{b}_i = {b}_i$, for $1\leq i \leq n$, since $p \not|\; m$. This contradicts the assumption that $\bm{\hat{b}} \neq \bm{{b}}$. Thus, all elements in a column are unique.
\end{proof}

Now we characterize the size of an equivalence class by using Lemma \ref{thm:uniqueness inside column} and results in the Appendix.

\begin{theorem}
\label{thm:LSM}
	Let $\mathcal E$ be an equivalence class of sets in $\left\{\mathbb Z_p^n\right\}_m$ and $S\in \mathcal E$.
    The number of distinct data sets in the $i^{th}$ column of $\mathcal{M}(S)$ is of the form $p^{r_i}$, $0\leq r_i\leq n$.
\end{theorem}

\begin{theorem}\label{thm:factors}
	Let $\mathcal E$ be an equivalence class of sets in $\left\{\mathbb Z_p^n\right\}_m$ and $S\in \mathcal E$. Then $|\mathcal E|$ has $p^s$ as a factor, where $s = \min(r_1,r_2,\ldots,r_{(p-1)^n})$ and the $r_i$'s are the same as those in Theorem  \ref{thm:LSM} (i.e. the number of distinct  sets in the $i^{th}$ column of $\mathcal{M}(S)$ is $p^{r_i}$, $0\leq r_i\leq n$).
\end{theorem}

\begin{proof}[Proof of Thms \ref{thm:LSM}, \ref{thm:factors}]
	Fix an input data set $S$ and consider column $\mathcal{M}(S,\bm{a})$ of the linear shift matrix $\mathcal{M}(S)$. According to Theorem  \ref{thm:same-repeat}, the sets in a column of a linear shift matrix appear an equal number of times. Let  $\alpha$ be the number of appearances of each set in the column $\mathcal{M}(S,\bm{a})$. Let $\gamma$ be the number of unique sets in the column. Then we have
	$$\gamma \alpha =p^n,$$
	which implies $\gamma=p^r$ for some $0\leq r\leq n$. As the size of each column of the linear shift matrix is in the form $p^r$ for some $0\leq r\leq n$, there exist $0\le r_1,r_2,\ldots,r_{(p-1)^n}\le n$ such that the number of distinct input data sets in each column is $p^{r_1},p^{r_2},\ldots, p^{r_{(p-1)^n}}$. It follows from Theorem  \ref{newthm:inter-columns} that the sizes of all equivalence classes have a factor $p^s$, where $s = \min(r_1,r_2,\ldots,r_{(p-1)^n}).$

\end{proof}

\begin{theorem}\label{TFAE}
Let $\mathcal E$ be an equivalence class of sets in $\left\{\mathbb Z_p^n\right\}_m$ and $S\in \mathcal E$. The following are equivalent:
\begin{enumerate}
    \item there exists $1\leq i\leq (p-1)^n$ such that $r_i=0$,
    \item $m=p^n$,
    \item $r_i=0$ for all $1\leq i\leq (p-1)^n$,
\end{enumerate}
where the $r_i$'s are the same as those as in Theorem  \ref{thm:LSM} (i.e. the number of distinct sets in the $i^{th}$ column of $\mathcal{M}(S)$ is $p^{r_i}$, $0\leq r_i\leq n$).

\end{theorem}
\begin{proof}
    (2) $\implies$ (3):
    Let $n \in \mathbb Z^+$ and $p$ be prime. Assume $m=p^n$. Then $\big|\{\mathbb Z_p^n\}_m\big| = \frac{m}{p^n}  =~1$, which implies that there is only one equivalence class of size one. Thus, there is only one distinct set in each column (i.e. $p^{r_i} = 1$). So, $r=0$ for all $1\leq i\leq (p-1)^n$.

    (3) $\implies$ (1) is straightforward.

    (1) $\implies$ (2):
	Assume $r_i=0$ for some $1\leq i\leq (p-1)^n$. That is, all entries of the $i$-th column $\mathcal{M}(S,\bm{a})$ are the same. Let $S'$ denote the unique set of the $i$-th column. Assume for the sake of contradiction that $m \neq p^n$. Then there exist two distinct elements $\bm{\alpha},\bm{\beta}\in\mathbb Z_p^n$ such that $\bm{\alpha} \notin S'$ and $\bm{\beta} \in S'$. By assumption, all sets of the $i$-th column are the same ($S'$). However, there exists $\bm{b} \in \mathbb Z_p^n$ such that $\bm{b} = \bm{\alpha} - \bm{\beta} \pmod p$. Then, the set $S''=S'+\bm{b}$ is also in the $i$-th column of $\mathcal{M}(S)$. However, note that $\bm{\alpha}=\bm{\beta}+\bm{b} \in S''$, which implies $S'' \neq S'$. Thus, there are at least two distinct sets in the $i$-th column, contradicting that all sets in the $i$-th column are the same. Therefore, $m=p^n$.
\end{proof}

\begin{theorem}
\label{thm:divisibility of p^n}
	Let $\mathcal E$ be an equivalence class. Let $p \not|\; m$. Then $|\mathcal E|$ has $p^n$ as a factor.
\end{theorem}

\begin{proof}
	Based on Lemma \ref{thm:uniqueness inside column} and Theorem \ref{thm:same-repeat}, the size of an equivalence class, which is the number of distinct entries in the linear shift matrix, is a multiple of $p^n$ since each column with $p^n$ entries either adds $p^n$ distinct entries (no repetitions with previous entries) or none (all entries in the column have already appeared). Thus, the sizes of each equivalence class has $p^n$ as a factor.
\end{proof}

For given $p$, $n$ and, $m$, the number of data sets is $\sum_{i=1}^{t} |\mathcal{E}_i| = \dbinom {p^n}{m}$, where $\mathcal{E}_i$, $i=1,\dots,t$ are the equivalence classes. If $m=p^n$, we have $t=1$ and the size of the unique equivalence class is 1. Otherwise (if $m<p^n$), according to Theorems \ref{thm:factors} and \ref{TFAE}, each $|\mathcal E_i|$ has $p$ as a factor, so $|\mathcal E_i| \geq p$ for all $1\leq i \leq t $. Thus, the number of equivalence classes is $t \leq \dbinom {p^n}{m}/p$, which provides an upper bound. In particular if $p \not|\; m$, then the number of  equivalence classes is $t \leq \dbinom {p^n}{m}/p^n$ by Theorem  \ref{thm:divisibility of p^n}. 

\begin{corollary}
When $m=p^n$, there is a unique equivalence class. When $m<p^n$, an upper bound for the number of equivalence classes is $\dbinom {p^n}{m}/p$.
Specifically, if $m<p^n$ and $p \not|\; m$, the upper bound for the number of equivalence classes reduces to be $\dbinom {p^n}{m}/p^n$.
\end{corollary}

\section{DoEMS: Linking Design of Experiments and Model Selection}
\label{sec:web-introduction}

To facilitate linking design of experiments and model selection, we built a database of all annotated equivalence classes of input data sets in $\left\{\mathbb Z_p^n\right\}_m$ for the cases $p=2$, $2\leq n\leq 4$, and $1\leq m\leq p^n$; and for the cases $p=3$, $n=2$, and $1\leq m\leq p^n$. The database is a linkage of two tables, described below.

The first table contains all equivalence classes of input data sets for each value of $p$, $n$, and~$m$; see the algorithm in Table \ref{alg1} in the Appendix. The equivalence classes are indexed by a unique identifier called a \texttt{classlabel}.  For each set $S$ in a row of the table (for fixed $p,n,m$), the columns are the \texttt{classlabel}, with which the other members of the equivalence class~$\mathcal E$ can be retrieved, and whether $S$ is the representative for $\mathcal E$; see the algorithm in Table \ref{alg2} in the Appendix.  Flagging input data sets which are representatives permits efficient querying of the resulting databases (see details below).

The second table contains all model bases associated to input data sets for each value of $p$,~$n$, and~$m$. For each set $S$ in a row of the table (for fixed $p,n,m$), the columns are the number of model bases associated to $S$ as well as the model bases represented as a nested set of standard monomials.  For completion's sake and for broader applicability to research involving computational algebraic geometry, we included the corresponding minimal generators of the leading term ideals and the reduced Gr\"obner bases for the ideal $I(S)$ of the points in $S$.  Given a set $S$, one reduced Gr\"obner basis $G$ for $I(S)$ (with respect to any monomial ordering) was computed using the ``Points'' package~\cite{PointsSource} in the computer algebra system Macaulay2~\cite{M2}.  Given $G$, the full set of Gr\"obner bases, leading term ideals, and sets of standard monomials were computed using the Macaulay2 ``gfanInterface'' package~\cite{gfanInterfaceSource}, which calls the software Gfan~\cite{gfan}.

The tables are linked via data sets to form a database of annotated equivalence classes.
The database is managed by SQLite~\cite{sql}, a C-language library that implements an efficient SQL database engine. The records in the SQLite table are organized by data sets and the corresponding fields are the columns from the previously generated tables (see Figure~\ref{fig:results-table}).
To access the database, we developed the website ``DoEMS: Linking Design of Experiments and Model Selection'' using PHP \cite{php} and which is publicly available at \texttt{https://s2.smu.edu/doems}. Finally, we use Python~\cite{python} to query the database. The flow chart in Figure \ref{fig:DoMESflow} illustrates computational paths in querying the database.

\begin{figure}[ht]
\centering
\includegraphics[width=5cm]{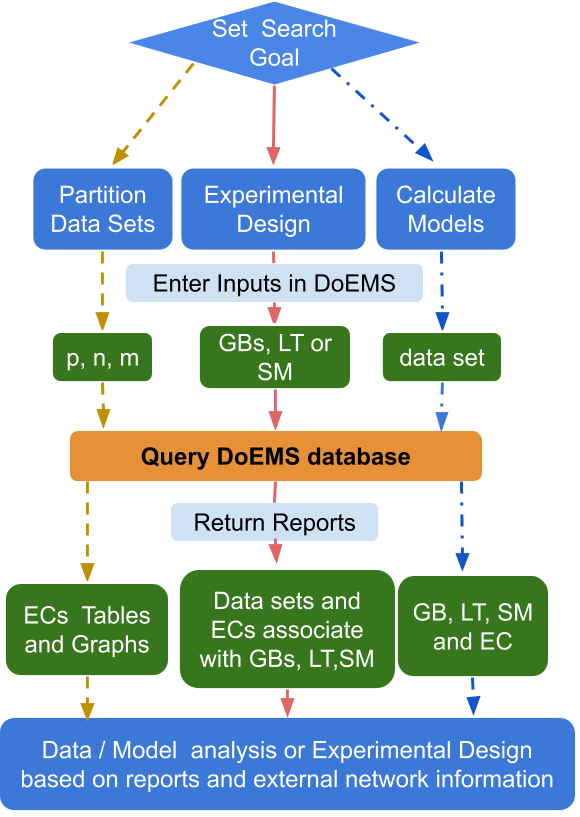}
\caption{Flowchart of computational paths with DoEMS.}
\label{fig:DoMESflow}
\end{figure}

\begin{table}[ht]
\begin{tabular}{c|l} \hline\hline
\multicolumn{2}{l}{\textbf{\large{Querying the Database}}}\\ \hline \hline
   \textit{Required inputs} & $p,n$ \\ \hline
   \textit{Optional inputs} & \textbf{data-centric}: $m$, data sets, is data set a representative?\\
   & \textbf{model-centric}: \#bases, bases, leading terms, Gr\"obner bases\\
   & \textbf{database-centric}: \texttt{classlabel} \\ \hline
\end{tabular}
\caption{Querying the database with two required inputs as well as data-, model-, and database-centric optional inputs.}
\label{query}
\end{table}

We implemented multiple ways to visualize the equivalence classes and the corresponding model bases; see Figure \ref{query}.
When only the number $p$ of states  and the number $n$ of nodes are given, the output is a table of all of the annotated equivalence classes; for example see Figure \ref{fig:results-table}.  When more inputs including the number $m$ of points are given, summary tables are provided. For example, Figure
\ref{fig:summary}  displays the number of equivalence classes, their representatives, the corresponding model bases, and the unique \texttt{classlabel} identifier for $p=3, n=2$, and $m=4$. Furthermore, we see that the 126 data sets of size 4 are partitioned in  7 equivalence classes.

\begin{figure}[ht]
\centering
\includegraphics[width=\textwidth]{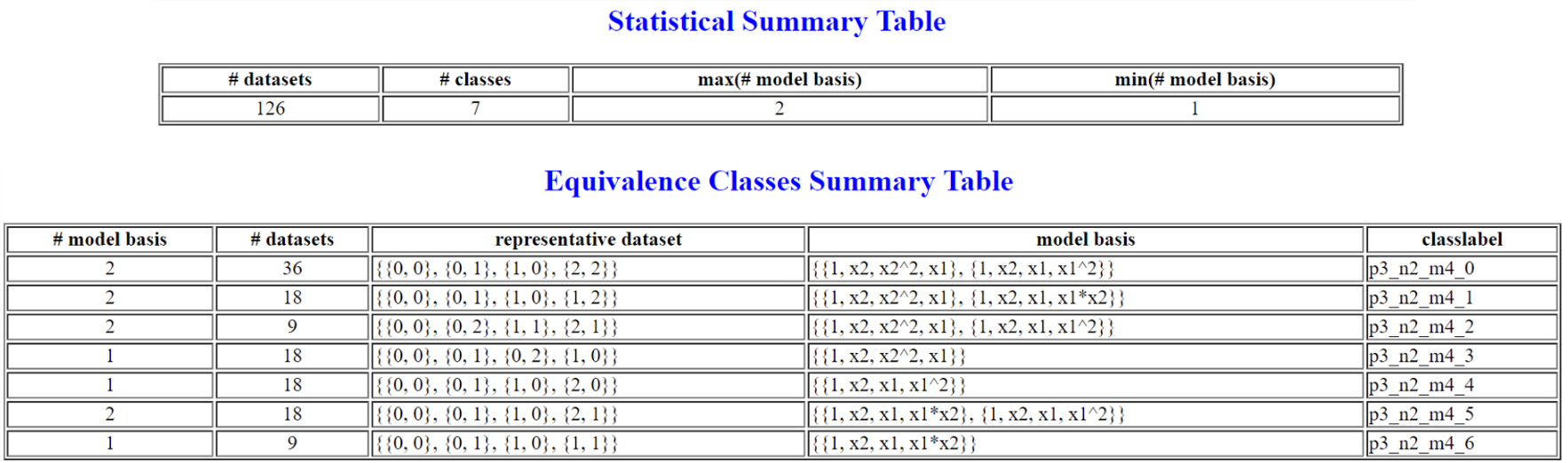}
\caption{Examples of two of the summary tables for $p=3, n=2, m=4$.  Displayed information includes the minimum and maximum number of bases among all equivalence classes, the number of equivalence classes, their representatives, the corresponding model bases, and the unique  \texttt{classlabel} identifier.}
\label{fig:summary}
\end{figure}

\begin{figure}[ht]
\centering
\includegraphics[width=0.5\textwidth]{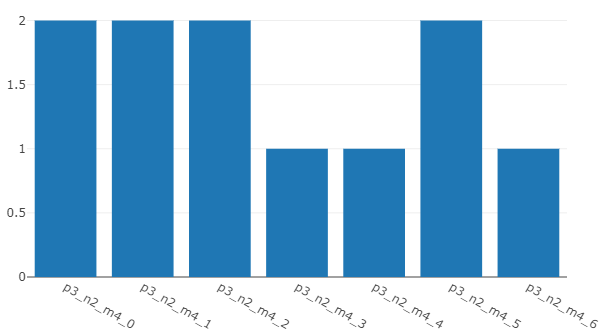}
\includegraphics[width=0.45\textwidth]{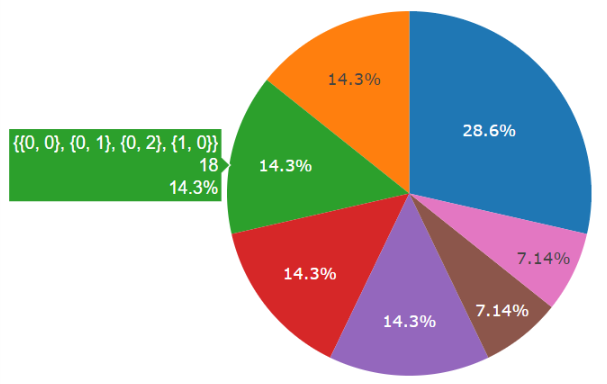}
\caption{Accompanying graphs for for $p=3, n=2, m=4$. Left: the number of model bases for each of the seven equivalence classes, arranged by the \texttt{classlabel}. Right: the relative sizes of the equivalence classes, along with their representatives (only one representative is displayed here).}
\label{fig:num-bases-ds}
\end{figure}

In addition, a bar graph and a pie chart accompany the summary tables when $m$ is given: the former shows the number of bases associated with each equivalence class while the latter shows the relative sizes of the equivalence classes along with their representatives; see Figure \ref{fig:num-bases-ds}.

The final display of results is a table listing all of the equivalence classes for data sets of size $m$ (and that satisfy other optional input).  This table is focused on input data sets. In particular, for each input data set, the following information is displayed: the number of associated model bases, the bases, the \texttt{classlabel} for its equivalence class, and whether the input data set is the representative for the class.  For those interested in algebraic-geometry computations, the table includes the minimal generators for the leading term ideals and the associated reduced Gr\"obner bases.
\begin{figure}[ht]
\centering
\includegraphics[width=\textwidth]{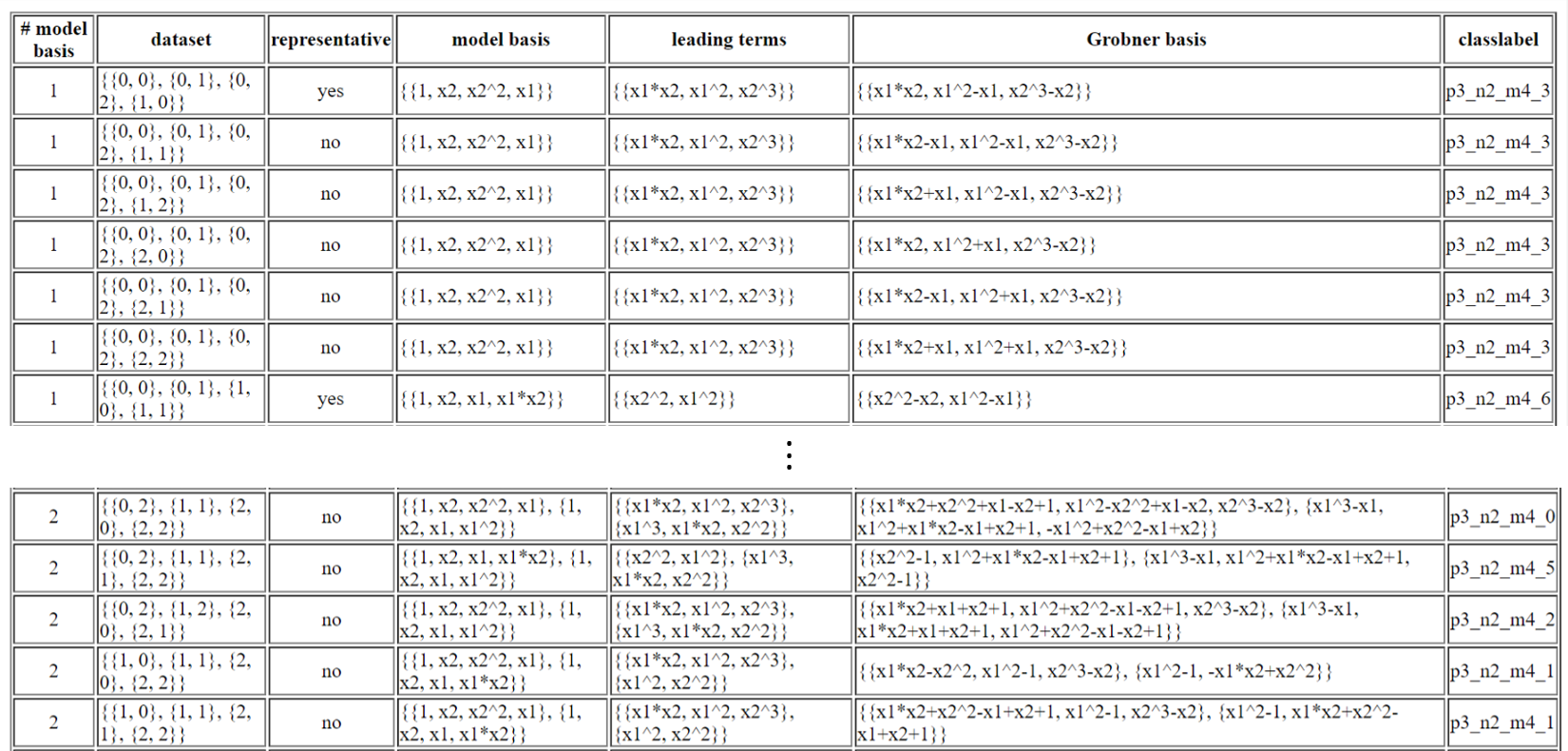}
\caption{A listing of all equivalence classes for $p=3, n=2, m=4$, arranged by input data set. For each input data set, included are the number of associated model bases, the bases, the \texttt{classlabel} for its equivalence class, and whether the input data set is the representative for the class.  Additionally, the minimal generators for the leading term ideals and the associated reduced Gr\"obner bases are provided.}
\label{fig:results-table}
\end{figure}

\section{Application: EGFR Inhibition Model on Tumor Growth}
\label{sec:applications}

In this section, we return to the two guiding questions posed at the beginning and demonstrate ways to address them in the context of a biological network. In particular we use observed interactions to design experiments  as well as use a given input data set to propose interactions which can be explained by the data.

In \cite{Steinway2016} the authors presented a Boolean model of a signaling network mediated by epidermal derived growth factor receptor (EGFR) and showed how an EGFR inhibitor suppresses tumor growth. We see from the wiring diagram in Figure \ref{fig:EGFRBoolean} that Rkip and Kras both directly affect Raf1, which in turn affects Proliferation. Here we will focus on the effect of the variables on the terminal node Proliferation and not consider the regulatory effects of the parameters.

\begin{figure}[ht]
\centering
\includegraphics[width=3.25cm]{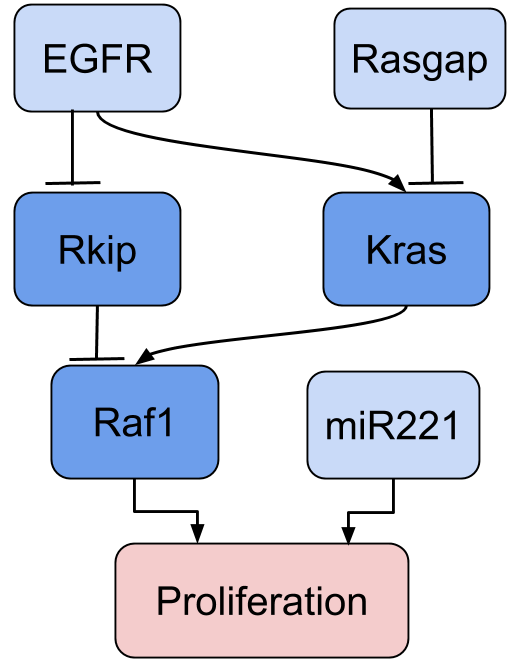}
\caption{Wiring diagram for the EFGR model. The model includes three parameters (EGFR, Rasgap, and miR221), and four variables (Rkip, Kras, Raf1, and Proliferation).  The variable Proliferation is a proxy for tumor growth. }
\label{fig:EGFRBoolean}
\end{figure}

\subsection{Using Interactions to Select Data}

We address the question of selecting data which identify a given interaction.
While the regulation by Rkip and Kras on Raf1 may be independent or coordinated, for the sake of the example, we assume the regulation is coordinated; so we have a known interaction between  Rkip and Kras.   In an effort to simplify the notation, we make the following substitutions: variables $x_1:=$ Rkip, $x_2:=$ Kras, $x_3:=$ Raf1, and $x_4:=$ Proliferation; and parameters $E:=$ EGFR, $R:=$ Rasgap, and $M:=$miR221. So the interaction between  Rkip and Kras is represented by the monomial $x_1x_2$. We aim to discover which input data sets identify the interaction $x_1x_2$.

The foundation for the presented results is a collection of equivalence classes of input data sets of a fixed size, where the equivalence classes are annotated by model bases and represented by input data sets in standard position.  We start with the monomial $x_1x_2$ being in a model basis~$B$.  Given that model bases must be closed as they are sets of standard monomials, then any monomial which divides $x_1x_2$ must be included. So we have that $B=\{1, x_1, x_2, x_1x_2\}$.  We also include $x_3$ in $B$ to allow for regulation by Raf1, resulting in $B=\{1, x_1, x_2, x_3, x_1x_2\}$. Since we have a model basis consisting of 5 monomials, we know we seek an input data set with 5 points.

Given $p=2$, $n=3$, and $m=5$, we know that there are $\binom{2^3}{5}=56$ possible sets of 5 points in~$\mathbb Z_2^3$.  Searching the database in DoEMS, we find that these~56 data sets are partitioned into~7 equivalence classes.  Refining the search further by including $x_1x_2$ in a model basis results in~32 data sets partitioned into~4 equivalence classes, seen below.

\begin{table}[ht]
\begin{center}
 \begin{tabular}{|c|c|c|c|}
 \hline
\#Bases &	Model Bases & Set in Standard Position & \#Sets	
\\ \hline
1&		$\{1, x_3, x_2, x_1, x_1x_2\}$ & $\{000,001,010,100,110\}$ & 8\\	\hline

2&		$\{1, x_3, x_2, x_1, x_1x_2\}$ &	$\{000,001,011,100,110\}$ & 8\\

&	 $\{1, x_3, x_2, x_2x_3, x_1\}$ &
 &\\	\hline

2&	$\{1, x_3, x_2, x_1, x_1x_2\}$ &	$\{000,001,010,101,110\}$ & 8\\

&	$\{1, x_3, x_2, x_1, x_1x_3\}$ & &	\\	\hline

&	$\{1, x_3, x_2, x_1, x_1x_2\}$ & & \\

3&	$\{1, x_3, x_2, x_1, x_1x_3\}$ & $\{000,001,010,100,111\}$ & 8	\\	

&	$\{1, x_3, x_2, x_2x_3, x_1\}$ & &	\\	\hline

\end{tabular}
\end{center}
\caption{The four equivalence classes associated with the model basis $\{1, x_1, x_2, x_3, x_1x_2\}$. Each blocked row represents an equivalence class.  The input data sets listed in the third column are the representatives of the equivalence classes and are written with commas and parentheses suppressed. The last column indicates the number of sets in each equivalence class.}
\label{table:equsEGFR_example1}
\end{table}

Notice that the basis $B$ is uniquely identified by the input data set $S_1=\{000, 001, \allowbreak 010, 100, \allowbreak 110\}$ in the first row of Table \ref{table:equsEGFR_example1}, as well as seven other sets which are linear shifts of $S_1$; we denote this equivalence class as $\mathcal E_1$.  For example $S'_1=\{111,110,101,011,001\}$ is another set in $\mathcal E_1$ as the mapping $\phi=(x+1,y+1,z+1):S_1\rightarrow S'_1$ is a linear shift.  In fact $S'_1$ has the largest set distance among the members of $\mathcal E_1$, whereas $S_1$ has the smallest set distance and thereby is the representative. An interpretation is that $S_1$ represents the set of experiments with the fewest active nodes while $S'_1$ has the most active nodes, as shown in Figure \ref{fig:experimentdesignboolean}.

\begin{figure}[ht]
\begin{center}
\includegraphics[width=14cm]{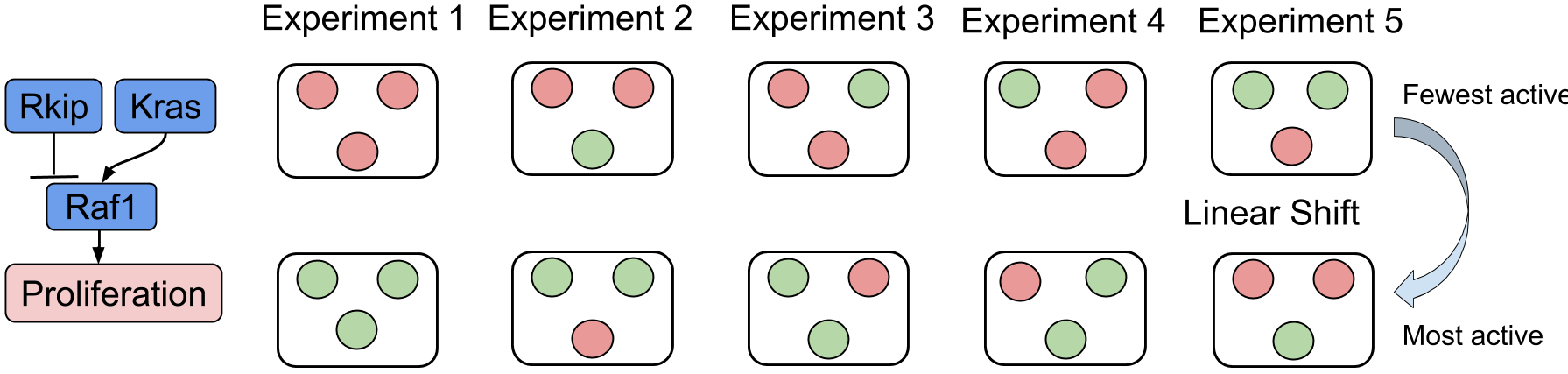}
\caption{Two experimental designs of the EGFR network. The top row depicts the input data set corresponding to the representative $S_1$ of the equivalence class $\mathcal E_1$, and the second row corresponds to the input data set $S'_1=\{111,110,101,011,001\}\in \mathcal E_1$.  The colored circles are in the same configuration as the variables in the subnetwork on the left; green represents 1 (active) and red represents 0 (inactive).}
\end{center}
\label{fig:experimentdesignboolean}
\end{figure}

\subsection{Using Data to Select Interactions}

Next we consider the second guiding question of which interactions are identifiable by a given set of input data.

The input data set $S_2=\{000,001,011,100,110\}$ in the second equivalence class in Table \ref{table:equsEGFR_example1} also identifies $B$, though not uniquely: $S_2$ also identifies the basis $\{1, x_3, x_2, x_2x_3, x_1\}$.  In fact, the monomial $x_2x_3$, representing the interaction between Kras and Raf1, is identified by the same data set as $x_1x_2$. Similarly, the input data set in the last equivalence class additionally identifies the monomial $x_1x_3$, corresponding to the interaction between Rkip and Raf1.

Now we consider the case of adding a point to an existing input data set.  Suppose we start with the set $\{000,001,010,100,110\}$, which is associated with the unique model basis $\{1, x_3, x_2, x_1, x_1x_2\}$. Which monomials are identified by adding one point to the set?  As there are only 3 possible points that could be added, we find the following results.

\begin{table}[ht]
\begin{center}
 \begin{tabular}{|c|c|c|}
\hline
\#Bases&	Model Bases & Data set	 \\ \hline
1& $\{1, x_3, x_2, \mathbf{x_2x_3}, x_1, x_1x_2\}$&	$\{000, 001, 010, \textbf{011}, 100, 110\}$	 \\ \hline
1& $\{1, x_3, x_2, x_1, \mathbf{x_1x_3}, x_1x_2\}$&	$\{000, 001, 010, 100, \textbf{101}, 110\}$  \\ \hline	
2&$\{1, x_3, x_2, \mathbf{x_2x_3}, x_1, x_1x_2\}$ &	$\{000,001,010,100,110, \textbf{111}\}$   \\ 	
&$\{1, x_3, x_2, x_1, \mathbf{x_1x_3}, x_1x_2\}$ &	   \\ \hline	

\end{tabular}
\end{center}
\caption{Effect of adding a point to an existing input data set.  Given the set $\{000,001,010,100,110\}$  with associated unique model basis $\{1, x_3, x_2, x_1, x_1x_2\}$, the table shows which new monomials (bold) are identified when a new point (bold) is added.}
\label{table:boolean-EGFR-add}
\end{table}

Table \ref{table:boolean-EGFR-add} shows the effect on model bases when a point is added to the input data set.  In particular, we see that adding the point 111 results in 2 distinct model bases, each of which contains a new monomial: $x_2x_3$ and $x_1x_3$, respectively. Thus these model bases predict more interactions beyond the one of interest.

\section{Discussion}
\label{sec:discussion}

The proposed computational framework is a data-driven approach for systematic and efficient experimental design and model selection as one process rather than independent steps in the data science pipeline. Performing these steps in a unified manner ensures economical experimental design where only the necessary experiments are performed as well as minimizing the number of computations needed for model selection based on data. While our work was primarily driven by problems in biological data science, the process can be applied to other fields where data collection capability is limited due to cost, time, ethical, or other constraints. Areas of impact include biomedical research, experimental physics, and real-time decision making. For example, while human biological samples such as cells, tissues, organs, blood, and sub-cellular materials are central for biomedical research, there are considerable ethical challenges in the collection, export, storage, and reuse of these samples~\cite{aarons14,tindana14}. As a second example, if an experiment corresponds to a strategy and the model bases represent outcomes,  knowing which strategies correspond to unique outcomes may be desirable in high-stakes operations. In many applications, it is crucial to only collect data that are necessary for the modeling process and information extraction.

The presented work suggests several theoretical and computational questions worth investigating. For example, the equivalence classes induced by a linear shift have interesting mathematical properties in terms of their number and sizes. These properties and their implications on the relationship between data and modeling have the potential to inform experimental design and model selection further.

Furthermore, the database we provide for quick lookup enables experimentalists of various backgrounds to use our results by allowing the circumventing of the computational steps. Growing the database and expanding the information it provides will make our approach even more accessible to the data science community.

\bibliographystyle{plain}
\bibliography{refs-for-arxiv}

\section{Appendix}
\label{apdx:div}

\subsection{Additional Results on Equivalence Classes}

We provide some additional results and proofs which are necessary for establishing the theorems in Section \ref{sec:div-eq}. Throughout this section, let $S$ be an input data set and $\mathcal{M}(S)$ be the linear shift matrix of $S$ as in Definition \ref{def:LSM}.

\begin{lemma} \label{newthm:inter-columns}
If there are two equal entries in any two distinct columns of $\mathcal{M}(S)$, then the two columns contain the same data sets.
\end{lemma}
\begin{proof}
	Let $m \leq p^n$ and $S_i=\{(x_{11},\ldots,x_{1n}),\ldots,(x_{m1},\ldots,x_{mn})\}$.
	Let $\bm{a}=(a_1,\ldots,a_n)$, $\bm{b} =(b_1,\ldots,b_n)$ and $\bm{\hat{a}} = (\hat{a}_1,\ldots,\hat{a}_n),\bm{\hat{b}} = (\hat{b}_1,\ldots,\hat{b}_n)$, where $a_i, \hat{a}_i \in \mathbb Z_p \backslash \{0\}$ and $b_i, \hat{b}_i \in \mathbb Z_p$, $1\leq i  \leq n$, such that $\mathcal{M}(S,\bm{a},\bm{b})=\mathcal{M}(S,\bm{\hat{a}},\bm{\hat{b}})$. By definition, we have
	\begin{align*}
	\mathcal{M}(S,\bm{a},\bm{b})& =  \{(a_1x_{11}+b_1,\ldots,a_nx_{1n}+b_n) , \ldots, (a_1x_{m1}+b_1,\ldots,a_nx_{mn}+b_n)\}\\
	& =\{(\hat{a}_1x_{11}+\hat{b}_1,\ldots,\hat{a}_nx_{1n}+\hat{b}_n) , \ldots, (\hat{a}_1x_{m1}+\hat{b}_1,\ldots,\hat{a}_nx_{mn}+\hat{b}_n)\}\\
	& =\mathcal{M}(S,\bm{\hat{a}},\bm{\hat{b}})
	\end{align*}
	
	Then for any $\bm{c} = (c_1,\ldots,c_n)\in \mathbb{Z}_p^n$, we have
	\begin{align*}
	\mathcal{M}(S,\bm{a},\bm{b}+\bm{c})& =  \{ \ldots, (a_1x_{i1}+b_1+c_1,\ldots,a_nx_{in}+b_n+c_n),\ldots\}\\
	& =\{ \ldots, (\hat{a}_1x_{i1}+\hat{b}_1+c_1,\ldots,\hat{a}_nx_{in}+\hat{b}_n+c_n),\ldots\}\\
	& =\mathcal{M}(S,\bm{\hat{a}},\bm{\hat{b}}+\bm{c}).
	\end{align*}

	As $\bm{c}$ is arbitrary, all the $p^n$ data sets in the column $\mathcal{M}(S,\bm{a})$ repeat in the column $\mathcal{M}(S,\hat{\bm{a}})$.
\end{proof}

\begin{corollary}
For any two columns $\mathcal{M}(S,\bm{a})$ and $\mathcal{M}(S,\bm{a'})$ of $\mathcal{M}(S)$, exactly one of the following is true:
\begin{enumerate}
    \item $\mathcal{M}(S,\bm{a})$ and $\mathcal{M}(S,\bm{a'})$ contain exactly the same data sets.
    \item $\mathcal{M}(S,\bm{a})$ and $\mathcal{M}(S,\bm{a'})$ contain no common data set.
\end{enumerate}
\end{corollary}

\begin{definition}
\label{def:cluster}
	For any $\bm{v}\in \mathbb Z_p^n\setminus \{\bm{0}\}$, define $C_{\bm{v}}=\{\bm{v},2\bm{v},\ldots,(p-1)\bm{v}\}$.
\end{definition}
	
Note that if $\bm{u}=k\bm{v}$ for some $1\leq k\leq p-1$, we will have $C_{\bm{u}}=C_{\bm{v}}$. Thus, we can pick any one of $\bm{v}, \ldots, (p-1)\bm{v}$ to represent $C_{\bm{v}}$, which implies the following proposition.

\begin{proposition}
		Let $A=\frac{p^n-1}{p-1}=p^{n-1}+p^{n-2}+\cdots +p+1$. We can pick $\bm{v}_1, \bm{v}_2, \ldots, \bm{v}_A \in \mathbb Z_p^n\setminus \{\bm{0}\}$ so that $C_{\bm{v_1}}, C_{\bm{v_2}},\ldots, C_{\bm{v_A}}$ form a partition of $\mathbb Z_p^n\setminus \{\bm{0}\}$.
\end{proposition}

\begin{proof}
    The total number of elements is $p^n-1$ and each equivalence class $C_v$ contains $p-1$ elements. This leads to the number of equivalence classes being $A=\frac{p^n-1}{p-1}$, which provides a partition of $\mathbb Z_p^n\setminus \{\bm{0}\}$.
\end{proof}

\begin{lemma}
\label{thm:twocases}
	For any data set $S$ and any $C_{\bm{v}}=\{\bm{v},2\bm{v},\ldots,(p-1)\bm{v}\}$, exactly one of the following is true:
	\begin{enumerate}
		\item $S+C_{\bm{v}}:=\{S+\bm{v},S+2\bm{v},\ldots,S+(p-1)\bm {v}\}$ along with $S$ are $p$ different points sets,\label{casei}
		\item $S+k\bm{v}=S$ for all $1\leq k\leq p-1$.\label{caseii}
	\end{enumerate}
\end{lemma}

\begin{proof}

We consider two cases.
	
If $S+k_1\bm{v}\neq S+k_2\bm{v}$ for any $0\leq k_1< k_2\leq p-1$, then Case \ref{casei} holds.
		
If there exists $0\leq k_1< k_2\leq p-1$ such that $S+k_1\bm{v}= S+k_2\bm{v}$, then we have $S=S+(k_2-k_1)\bm{v}$, which implies for any $1\leq k\leq p-1$, we have $S+k(k_2-k_1)\bm{v}=S+(k_2-k_1)\bm{v}+(k-1)(k_2-k_1)\bm{v}=S+(k-1)(k_2-k_1)\bm{v}=\cdots=S+(k_2-k_1)\bm{v}=S$. Hence Case \ref{caseii} holds.
\end{proof}

\begin{lemma}
 \label{thm:same-repeat}
	Given a column in $\mathcal{M}(S)$, a data set will appear an equal number of times.
\end{lemma}

\begin{proof}
Let $f(S',\bm{a})$ be the number of times that the data set $S'$ appears in the c column $\bm{a}=(a_1,\ldots,a_n)$ of $\mathcal{M}(S)$. Pick a data set $X$ such that
$$f(X,\bm{a}) = \max_{S'} f(S',\bm{a}).$$
Let $Y$ be a data set in the same column as $X$ and $Y \neq X$. Note that if we cannot find such~$Y$, then the column has only one unique data set and the lemma is trivially true. Otherwise, the lemma will follow if we can prove that $f(X,\bm{a}) \leq f(Y,\bm{a})$.

Pick $\bm{w} \in \mathbb Z_p^n\setminus \{\bm{0}\}$ such that $X + \bm{w} = X$. (Note that if we cannot find such $\bm{w}$, then the column has $p^n$ different data sets and the lemma follows immediately.) According to Definition \ref{def:cluster}, we have $C_{\bm{w}}=\{\bm{w},2\bm{w},\ldots,(p-1)\bm{w}\}$. By Theorem \ref{thm:twocases}, $X + \bm{z} = X$, for any $\bm{z} \in C_{\bm{w}}$. Assume $\alpha_{X}$ is the total number of $C_{\bm{w}}$. Then, $f(X,\bm{a}) = \alpha_{X}(p-1)+1$. Pick $\bm{v} \in \mathbb Z_p^n\setminus \{\bm{0}\}$ such that $X + \bm{v} = Y$.
(Note that such $\bm{v}$ exists since the data sets $X$ and $Y$ are in the same column). Then, for each $C_{\bm{w}}$, there is a set, $D = \{\bm{v+w},\bm{v}+2\bm{w}, \ldots, \bm{v}+(p-1)\bm{w}\}$ such that $X + \bm{d} = Y$, for all $\bm{d} \in D$. As $\big| D \big| = p-1,$ we can construct $p-1$ maps from $X$ to~$Y$ for each $C_{\bm{w}}$. Then note that there are $\alpha_X$ number of $C_{\bm{w}}$ and we construct $\alpha_X(p-1)+1$ maps from $X$ to $Y$ (by including $\bm{v}$ itself).

We will show that the constructed $\alpha_X(p-1)+1$ maps from $X$ to $Y$ are distinct.
	Note that the $\alpha_X(p-1)+1$ maps we constructed are $\bm{v}+\bm{z}$ for any $\bm{z}$ such that $X+\bm{z}=X$.
	
Let $\bm{w}\neq \bm{0}$ with $X+\bm{w}=X$ and $1\leq m_1<m_2\leq p-1$. If $\bm{v}+m_1w = \bm{v}+m_2\bm{w}$, then we have $m_1\bm{w}=m_2\bm{w}$, which implies $m_1=m_2$, a contradiction. Now let $\bm{w}, \bm{w}'\neq \bm{0}$ with $X+\bm{w}=X$, $X+\bm{w}'=X$,  $\bm{w}'\notin C_{\bm{w}}$, and $1\leq m_1, m_2\leq p-1$. If $\bm{v}+m_1w= \bm{v}+m_2\bm{w}'$, then we have $\bm{w}'=\frac{m_1}{m_2}\bm{w}$, which implies that $\bm{w}'\in C_{\bm{w}}$, a contradiction. %
%
%
	Hence the maps $\alpha_X(p-1)+1$ are distinct.

This implies the data set $f(Y,\bm{a}) \geq \alpha_X(p-1)+1 = f(X,\bm{a})$. Moreover, recall that $f(X,\bm{a}) = \max_{S'} f(S',\bm{a})$, we have $f(Y,\bm{a}) = f(X,\bm{a})$. Thus, each data set repeats the same number of times in one column of the linear shift matrix.
\end{proof}

\newpage
\subsection{Algorithms}
Below are the primary algorithms used to generate the equivalence classes table in DoEMS.
\begin{table}[h]
\begin{tabular}{l|l} \hline\hline
\multicolumn{2}{l}{\textbf{\large{Algorithm: Generate Equivalence Classes}}}\\ \hline \hline
   \textit{Inputs} & $p,n,m$\\ \hline
   \textit{Outputs} & All equivalence classes of sets in $\mathcal S = \{\mathbb Z_p^n\}_m$  \\   \hline
   \textit{Steps}
& Generate  all data sets in $\mathcal S$\\
& Select one data set $S$ in $\mathcal S$\\
& Initialize equivalence class set $\mathcal E=[S]$ and set $S_{rest} = \mathcal S \setminus [S]$\\
& \\
& \emph{\#Generate all linear shifts $\ell=ax+b$ for one coordinate}\\
& Initialize linear shift list $L = [\ ]$\\
& for $a$ in $[1, \ldots, p-1]$ \\
&\hspace{1cm} for $b$ in $[0,\ldots, p-1]$\\
&\hspace{2cm} append $[a, b]$ to $L$\\
& $LS =[[a_1, b_1],\ldots, [a_n, b_n]$ : $0 < a_i < p $, $ 0 \leq b_i < p$, $1\leq i\leq n]$ \\
& \\
& \emph{\# Generate all combinations of linear shifts for all coordinates}\\
& while $S \neq \emptyset $\\
& \hspace{1cm} for $f_i$ in $LS = [[a_1, b_1],\ldots, [a_n, b_n]]$\\
&\hspace{2cm}  apply $\ell_i$ to $S$ to generate new data set $S^*$ \\
&\hspace{2cm}  append $S^*$ to $\mathcal E$ and remove $S^*$ from $S_{rest}$\\
&\hspace{1cm}  set $S$ to first element in $S_{rest}$ and set $\mathcal E = [S]$ \\
& return $\mathcal E$ \\ \hline
\end{tabular}
\caption{Algorithm to generate all equivalence classes for data sets in $\{\mathbb Z_p^n\}_m$.}
\label{alg1}
\end{table}

\begin{table}[h]
\begin{tabular}{l|l} \hline\hline
\multicolumn{2}{l}{\textbf{\large{Algorithm: Find Representatives}}}\\ \hline \hline
   \textit{Input} & an equivalence class $\mathcal E\subset \mathcal S = \{\mathbb Z_p^n\}_m$\\ \hline
   \textit{Output} & a representative $S\in \mathcal S$ for $\mathcal E$; its set distance  \\   \hline
   \textit{Steps}
& Initialize representatives list $rep=[ \ ]$\\
& Initialize $rep=\mathcal E[0]$; $D=+ \infty$\\
& for $S$ in $\mathcal E$\\
&\hspace{1cm} Set $D_{new} =D(S,0)$\\
&\hspace{1cm} if  $D_{new} < D$\\
&\hspace{2cm}  $D = D_{new}$\\
&\hspace{2cm}  $rep = S$\\
&  return $S, D$\\ \hline
\end{tabular}
\caption{Algorithm to identify the representative of an equivalence class.}
\label{alg2}
\end{table}

\end{document}